\documentclass[reqno,dvipsnames]{amsart}

\usepackage[a4paper,hmarginratio=1:1]{geometry}

\usepackage{comment} 
\usepackage{pdflscape}
\usepackage{hyperref}
\hypersetup{
	colorlinks=true,
	citecolor=blue,
	linkcolor=blue,
	filecolor=magenta,      
	urlcolor=cyan,
}

\usepackage{caption}
\usepackage{fullpage}

\usepackage{pgfplots}
\usepgfplotslibrary{fillbetween}

\usepackage[nameinlink,capitalise,noabbrev]{cleveref}
\usepackage{amsmath, amsthm, amssymb, graphicx, dsfont, stmaryrd, extarrows, mathtools}
\usepackage[T1]{fontenc}
\usepackage{textcomp}
\usepackage{lmodern}
\usepackage{microtype}

\usepackage[color,matrix,arrow,all]{xy}
\providecommand{\1}{}
\renewcommand{\1}{\mathbbm{1}}
\newcommand{\unit}{\mathbbm{1}}

\definecolor{internationalkleinblue}{rgb}{0.0, 0.18, 0.65}

\usepackage{enumitem}
\usepackage[hang,flushmargin]{footmisc}
\usepackage{amscd}
\usepackage{mathrsfs}
\usepackage{tikz-cd}
\usetikzlibrary{matrix,patterns}
\usepackage{color}
\usepackage{bbm}
\usepackage[mathscr]{euscript}

\numberwithin{equation}{section}
\theoremstyle{plain}
\newtheorem{thm}[equation]{Theorem}

\newtheorem*{thm*}{Theorem}
\newtheorem*{prop*}{Proposition}
\newtheorem*{mthm*}{Meta Theorem}
\newtheorem*{cor*}{Corollary}

\newtheorem{prop}[equation]{Proposition}

\newtheorem{strat}[equation]{Strategy}
    
\newtheorem{cor}[equation]{Corollary}       
\newtheorem{lem}[equation]{Lemma}

\theoremstyle{definition} 
\newtheorem{defn}[equation]{Definition} 
\newtheorem{hyp}[equation]{Hypothesis}
\newtheorem{ex}[equation]{Example}

\newtheorem{rem}[equation]{Remark}

\newcommand{\C}{\mathscr{C}}
\newcommand{\D}{\mathscr{D}}

\newcommand{\E}{\mathscr{E}}
\newcommand{\G}{\mathcal{G}}

\newcommand{\K}{\mathcal{K}}

\newcommand{\Sp}{\mathrm{Sp}}

\newcommand{\msf}[1]{\mathsf{#1}}

\renewcommand{\phi}{\varphi}

\newcommand{\Q}         {{\mathbb{Q}}}
\newcommand{\Z}         {{\mathbb{Z}}}

\newcommand{\Hom}       {\operatorname{Hom}}

\newcommand{\cK}{\mathcal{K}}

\newcommand{\sfD}{\mathsf{D}}

\renewcommand{\bmod}[1]{\mathrm{Mod}(#1)}
\newcommand{\mrm}[1]{\mathrm{#1}}

\newcommand{\cpctRecollement}[3]{
\xymatrix@C=2em{{#1} \ar[r]|-{#2} & {#3}}
}

\title{Tensor-triangular rigidity in chromatic homotopy theory}
\author{Scott Balchin}
\address[Balchin]{Mathematical Sciences Research Centre, Queen's University Belfast, UK}
  \email{s.balchin@qub.ac.uk}

\author{Constanze Roitzheim}
\address[Roitzheim]{University of Kent \\ School of Mathematics, Statistics and Actuarial Science\\ Cornwallis Building, Canterbury, CT2 7NF}
\email{c.roitzheim@kent.ac.uk}

\author{Jordan Williamson}
\address[Williamson]{Department of Algebra, Faculty of Mathematics and Physics, Charles University in Prague, Sokolovsk\'{a} 83, 186 75 Praha, Czech Republic}
\email{williamson@karlin.mff.cuni.cz}

\usepackage[colorinlistoftodos]{todonotes}

\pgfplotsset{%
    /pgfplots/area legend/.style={
        /pgfplots/legend image code/.code={
            \fill[##1] (0cm,0.6em) rectangle (2*\pgfplotbarwidth,-0.3em);
}, },
}
\pgfplotsset{compat=1.17}

\DeclareMathOperator{\whodot}{\widehat{\odot}}
\DeclareMathOperator{\whotimes}{\widehat{\otimes}}

\begin{document}

\maketitle

\begin{abstract}
	We study the uniqueness of enhancements of tensor-triangulated categories. To do so, we provide conditions under which these enhancements interact well with categorical decompositions. As an application we obtain new results about the uniqueness of enhancements in chromatic homotopy theory.
\end{abstract}

\let\thefootnote\relax
    
\section{Introduction}

Triangulated categories, since their inception by Verdier~\cite{verdier} and Dold--Puppe~\cite{DoldPuppe}, have formed a central component of pure mathematics with their influence being seen throughout algebra, representation theory, and homotopy theory. Unfortunately, triangulated categories have well-known shortcomings, such as poor functoriality, lack of (co)completeness, and undesirable behaviour with respect to ring objects. These limitations can be resolved by working with an enhancement of the triangulated category, that is, by having a stable $\infty$-category whose homotopy category is the triangulated category in question. However, this begs a natural question: is this passage to an enhancement unique? This is precisely the question that rigidity addresses.  If there is a unique enhancement, we say that the enhancement is \emph{rigid}. Conversely, an enhancement which is not unique is called \emph{exotic}.


 Many of the natural examples of triangulated categories come with further structure, namely they have a compatible symmetric monoidal structure. It is reasonable to investigate how this additional structure interacts with the question of rigidity. That is, we now require our enhancement to be stable and monoidal. The question of \emph{tensor-triangular rigidity} considers the uniqueness of such enhancements. The uniqueness of the enhancement can be interpreted in a hierarchy of strengths depending on how much structure is preserved. We will make this hierarchy explicit in \cref{defn:difrigid}.

 
 %
 
 

To study tensor-triangular rigidity we take a leaf from the book of chromatic homotopy theory. It is known via the chromatic convergence theorem that the  homotopy type of any finite spectrum may be recovered from its $E(n)$-localizations, and furthermore, that the $E(n)$-localization may be built from $E(n{-}1)$-localization and $K(n)$-localization. These fracturing techniques are in fact special cases of a more general theory appearing in tensor-triangular geometry~\cite{balmer,BalmerFavi}, and have been extensively used in, among other things, the modular representation theory of finite groups~\cite{BIK}. These objectwise decompositions can furthermore be extended to categorical decompositions, allowing the deconstruction of tensor-triangulated categories into smaller pieces which are more tractable. This approach has been fruitful in chromatic homotopy theory \cite{BAC}, and in equivariant homotopy theory \cite{adelicm, GStorus}.

If the smaller pieces arising in the aforementioned decompositions happen to be tensor-triangular rigid in a compatible way, one may hope to deduce tensor-triangular rigidity for the whole category by reversing the decomposition process. We make this strategy explicit in our first main theorem, \cref{metatheorem}. This theorem tells us that it is possible to reverse the process provided that a certain compatibility condition is satisfied.


Our main application of \cref{metatheorem} is in the setting of chromatic homotopy theory. To this end, let us first discuss the state of the art in rigidity results without taking into account the tensor structure. One landmark result is by Schwede who proved that the category of spectra, denoted here as $\Sp$, is rigid~\cite{Schwede2,Schwederigidity}. For localized categories of spectra, it has been proved that the categories $L_1\Sp_{(2)}$ of $E(1)$-local spectra and $L_{K(1)} \Sp_{(2)}$ of $K(1)$-local spectra, both at the prime $2$, are rigid~\cite{Roitzheimrigidity,Ishakrigidity}. 
Conversely, we have a wealth of examples which are not rigid. For example the category $L_n\Sp_{(p)}$ of $E(n)$-local spectra at the prime $p$ is exotic whenever $2p-2>n^2+n$~\cite{Franke, PatchkoriaPstragowski}.  This leaves a substantial  range in which the question of rigidity is unknown.



It is clear by definition that it is harder to be an exotic model in the tensor-triangular sense as we ask for more structure to be preserved. This phenomenon can be seen in the $E(n)$-local category of spectra, where work of Barkan~\cite{Barkan} provides tensor-triangular exotic models in the range  $2p-2 > n^2 + 3n$. On the other hand, this means that the extra structure afforded by the tensor provides a scaffolding to prove positive results regarding rigidity as we undertake in this paper. All in all, the situation of rigidity for $L_{n} \Sp_{(p)}$ can thus be summarised as in~\cref{fig:graph}.

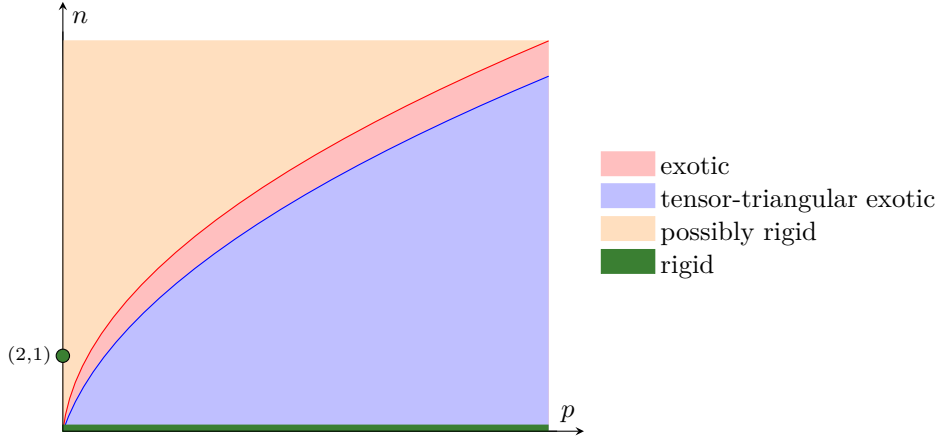
\begin{figure}[h]
\centering
\begin{tikzpicture}\pgfplotsset{ticks=none}
    \begin{axis}[legend pos=outer north east, legend style={draw=none}, area legend,
  xlabel = $p$,
  ylabel = $n$,
  axis lines = middle,
  ymax = 11,
  xmax = 60,
  xmin = 1,
  legend style={at={(1.7,0.5)},anchor=east},
  legend cell align={left}]
    
    \addplot[domain=0:10, color = red!25] (0.5*x^2 + 0.5*x + 1,x);
    \addlegendentry{exotic}
    \addplot[domain=0:9.095, color = blue!25] (0.5*x^2 + 1.5*x + 1,x);
    \addlegendentry{tensor-triangular exotic}
    \addplot[name path=xtop, domain=1:56, color = orange!25] {10};
    \addlegendentry{possibly rigid}
    \addplot[name path=placeholder, domain=-0:56, color = OliveGreen] {0};
    \addlegendentry{rigid}
    
    \addplot[domain=-0:9.095, color = blue, name path=A] (0.5*x^2 + 1.5*x + 1, x); 
    \addplot[domain=-0:10, color = red, name path=B] (0.5*x^2 + 0.5*x + 1, x);
    \addplot[name path=placeholder, domain=0:56, color = OliveGreen, line width = 5] {0};
    \addplot[draw=none,name path=x, domain=0.1:56] {0};

    \addplot+[orange!25] fill between[of=x and xtop, soft clip={domain=-0:100}];
    \addplot+[red!25] fill between[of=x and B, soft clip={domain=-0:100}];
    \addplot+[blue!25] fill between[of=x and A, soft clip={domain=-0:100}];
    \end{axis}
    \draw (0,0) -- (0,5.3);
    \draw (0,0) -- (6.5,0);
    \filldraw[fill = OliveGreen, stroke = black] (0,1) circle (2.5pt);
    \node[label={180:{${\scriptstyle(2,1)}$}}] at (0.1,1) {};
\end{tikzpicture}
\caption{A graph showing the various ranges of rigidity for $L_n \Sp_{(p)}$. We highlight that the $n=0$ line and the single point $p=2$, $n=1$ are the only known rigid examples.}\label{fig:graph}
\end{figure}

Our main contribution to this story is the following theorem, which links the tensor-triangular rigidity of the $E(n)$-local category to the tensor-triangular rigidity of the $K(i)$-local categories for $1 \leqslant i \leqslant n$ and follows from an application of  \cref{metatheorem}. In particular, the main technical contribution of \cref{sec:spectraex} is a verification of the aforementioned compatibility conditions required to assemble tensor-triangular rigidity from the local pieces, which in this case are the $K(i)$-local categories $L_{K(i)}\Sp_{(p)}$.

\begin{thm*}[\cref{main}, \cref{K(n)}]
Let $n \geqslant 1$. If $L_{K(i)}\Sp_{(p)}$ is unitally tensor-triangular rigid for all $1 \leqslant i \leqslant n$, then $L_n\Sp_{(p)}$ is strongly tensor-triangular rigid.
\end{thm*}

Along the way, we also prove general results relating the different  notions of tensor-triangular rigidity for localized categories of spectra, see \cref{subsec:specialfunctors}.

\subsection*{Conventions}
We will freely use the language of stable $\infty$-categories from~\cite{HTT,HA}. For a stable $\infty$-category $\C$, we will write $\Hom_\C^\Sp(-,-) \in \Sp$ for the mapping spectrum, and when $\C$ is moreover closed symmetric monoidal, we write $\Hom_\C^\C(-,-) \in \C$ for the internal hom object. That is, the subscript denotes the domain of the hom functor, and the superscript denotes the codomain.

By a \emph{localization} $L$ of a presentable stable symmetric monoidal $\infty$-category $\C$ we mean a functor $L \colon \C \to \D$ with a fully faithful right adjoint $\iota$ such that if $LX \simeq 0$ then $L(X \otimes Y)\simeq 0$ for all $X,Y \in \C$ (see \cite[Definition 3.1.1]{axiomatic}). We often abuse notation and write $L$ for the composite $\iota L$. For a localization $L$ we identify $\D$ with $L\C$, the full subcategory of $\C$ spanned by the $L$-local objects, that is, those objects $X$ for which the unit map $X \to LX$ is an equivalence. In \cite{HA} such localizations are called monoidal localizations, but we will follow the conventions of Hovey--Palmieri--Strickland~\cite{axiomatic}.

\subsection*{Acknowledgements}
SB would like to thank the Max Planck Institute for Mathematics for its hospitality, and was partially supported by the European Research Council (ERC) under Horizon Europe (grant No.~101042990).

CR thanks the LMS for an Emmy Noether Fellowship which supported a research visit for this project. 

JW was supported by the grant GA\v{C}R 20-02760Y from the Czech Science Foundation, and by the project
PRIMUS/23/SCI/006 from Charles University. 

SB and JW would also like to thank the Hausdorff Research Institute for Mathematics for its hospitality and support during the trimester program `Spectral Methods in Algebra, Geometry, and Topology', funded by the Deutsche Forschungsgemeinschaft (DFG, German Research Foundation) under Germany's Excellence Strategy – EXC-2047/1 – 390685813.

\newpage

\section{Types of rigidity}\label{sec:levels}

In this section we will introduce the concept of \emph{tensor-triangular rigidity} and illustrate it with some familiar examples. Let us begin by recalling the usual definition of rigidity.

\begin{defn}
A  presentable stable $\infty$-category $\C$ is \emph{rigid} if for any  presentable stable $\infty$-category $\D$ which is equipped with a triangulated equivalence $\Phi \colon h\C \xrightarrow{\simeq_{\Delta}} h\D$, there is an equivalence $F \colon \C \xrightarrow{\simeq} \D$ on the level of $\infty$-categories. 
\end{defn}

The philosophy is that if $\C$ is rigid, then its structure as a stable $\infty$-category is uniquely determined by the triangulated structure of its homotopy category. 
If $\C$ is not rigid, then its homotopy category $h\C$ has an \emph{exotic model}, that is, a  presentable stable $\infty$-category $\D$ for which $h\C$ is triangulated equivalent to $h\D$, but $\C$ is not equivalent to $\D$. 
A classical example of an exotic model is given by Schlichting~\cite{Schlichting} (with further work by Dugger--Shipley~\cite{DuggerShipley}) as we now recall. 
Let $k = \Z/p$, and consider the Frobenius rings $R = \Z /p^2$ and $R_{\epsilon} = k [\epsilon] / (\epsilon^2)$. Then the stable module categories of these rings have triangulated equivalent homotopy categories, but there is no equivalence of the level of the $\infty$-categories. 
Another example comes from the category of modules over the Morava $K$-theory spectrum $K(n)$, for whom $\mathrm{Mod}(K(n)_*)$ is an exotic model, see~\cite[Remark 2.5]{SchwedeShipleyuniqueness} or~\cite[\S 5.2]{algebraicity} for details.

\begin{rem}
The equivalence $F \colon \C \xrightarrow{\simeq} \D$ obtained via rigidity need not have any connection to the given equivalence $\Phi$. In particular, it need not be the case that $\Phi$ is the derived functor of $F$.
\end{rem}

The goal of this paper is to introduce and explore other forms of rigidity that take into account extra structure on the categories $\C$ and $\D$. The structure that we will be interested in is when $\C$ and $\D$ are moreover symmetric monoidal $\infty$-categories.

\begin{defn}\label{defn:difrigid}
Let $(\C, \otimes, \1_{\C})$ be a presentable stable closed symmetric monoidal $\infty$-category. Then we say that $\C$ is
\begin{itemize}
	\item \emph{tensor-triangular rigid} (henceforth \emph{tt-rigid}) if whenever there is a tensor-triangulated equivalence $\Phi \colon h \C \xrightarrow{\simeq_{\Delta, \otimes}} h \D$ for $(\D, \otimes, \1_{\D})$ a presentable stable closed symmetric monoidal $\infty$-category, there is an equivalence of $\infty$-categories $F \colon \C \xrightarrow{\simeq} \D$,
	\item \emph{unitally tt-rigid} if it is tt-rigid and the  equivalence $F$ satisfies $F(\1_{\C}) \simeq \1_{\D}$,
	\item \emph{strongly tt-rigid} if it is tt-rigid and the equivalence $F$ is symmetric monoidal.
\end{itemize}
\end{defn}

Immediately from the definitions, we see that the following implications hold.
\[
\xymatrix{
\text{strongly tt-rigid} \ar@{=>}[r] & \text{unitally tt-rigid} \ar@{=>}[r] & \text{tt-rigid} & \ar@{=>}[l] \text{rigid}
}
\]
Let us record our main guiding examples.

\begin{ex}\label{ex:unitally}
The categories $\Sp$, $L_1\Sp_{(2)}$ and $L_{K(1)}\Sp_{(2)}$ are rigid as proved by Schwede~\cite{Schwederigidity}, Roitzheim~\cite{Roitzheimrigidity}, and Ishak~\cite{Ishakrigidity} respectively. As such, they are tt-rigid. For localizations of spectra, we will see that some of the above implications have converses (c.f., \cref{subsec:specialfunctors}). As such, in this instance we can say more: they are all \emph{strongly} tt-rigid, see \cref{ex:spectra}.
\end{ex}

\begin{ex}\label{ex:specta}
Franke conjectured the existence of exotic models for the categories $L_n\Sp_{(p)}$ in certain ranges~\cite{Franke}. A succession of work culminating in Patchkoria--Pstr\k{a}gowski~\cite{PatchkoriaPstragowski} confirms this conjecture in the range $2p-2 > n^2 + n$. As such, in this range the category $L_n\Sp_{(p)}$ is not rigid. Moreover, the $E(n)$-local categories also provide examples of tt-exotic models. Barkan has proved that when $2p-2>n^2+3n$, the category $L_n\Sp_{(p)}$ is not tt-rigid~\cite{Barkan}.
\end{ex}

Our next family of examples may be viewed as an enhancement of~\cite[Theorem 6.9]{BarnesRoitzheim} to the tensor-triangular setting. We recall that a graded-commutative ring $A$ is \emph{intrinsically formal as a commutative DGA} if for any commutative DGA $B$ with $A \simeq H_\ast(B)$ as graded-commutative rings, we have that $B$ is quasi-isomorphic to $A$ via a zig-zag of maps of commutative DGAs.

\begin{thm}\label{formalhomotopy}
Let $R$ be a commutative $H\mathbb{Q}$-algebra such that $\pi_*R$ is intrinsically formal as a commutative DGA. Then $\bmod{R}$ is strongly tt-rigid.
\end{thm}
\begin{proof}
Suppose that $\Phi\colon h\bmod{R} \to h\D$ is a tensor-triangulated equivalence where $\D$ is a presentable stable closed symmetric monoidal $\infty$-category. Therefore, $h\D$ is compactly generated by its monoidal unit so by~\cite[Proposition 7.1.2.7]{HA}, there exists a commutative ring spectrum $\E$ such that $\D$ is symmetric monoidally equivalent to $\bmod{\E}$. Since $\pi_*(\E) = \pi_*(R)$ which is rational, it follows that $\E$ is a commutative $H\mathbb{Q}$-algebra (as $H\mathbb{Q}$ is the rational sphere spectrum). By Shipley's algebraicization theorem~\cite[Theorem 1.2]{ShipleyHZ} (see also~\cite[Theorem 7.2]{Williamsonflat} and~\cite[\S 7.1.2]{HA}), there are commutative DGAs $\Theta(\E)$ and $\Theta(R)$ such that $H_*(\Theta(\E)) = \pi_*(R) = H_*(\Theta(R))$, together with symmetric monoidal equivalences \[\bmod{\E} \simeq \bmod{\Theta(\E)} \quad \text{and} \quad \bmod{R} \simeq \bmod{\Theta(R)}.\] Since $\pi_*(R)$ is intrinsically formal as a commutative DGA, there is a quasi-isomorphism of commutative DGAs $\Theta(\E) \simeq \Theta(R)$. Therefore there are symmetric monoidal equivalences $\bmod{\Theta(\E)} \simeq \bmod{\Theta(R)}$ by extension and restriction of scalars. Combining all these, we have a symmetric monoidal equivalence $\D \simeq \bmod{R}$, and hence $\bmod{R}$ is strongly tt-rigid.
\end{proof} 

\begin{rem}
We note that the rational assumption on $R$ in the previous result is needed to ensure that the corresponding DGA $\Theta R$ is commutative. A commutative $HA$-algebra (for a commutative ring $A$) corresponds to an $E_\infty$-$A$-DGA~\cite{RichterShipley}, and in general this cannot be rectified to a strictly commutative DGA.
\end{rem}

\begin{rem}
All of the examples of tt-rigid categories that we encounter in this paper are in fact strongly tt-rigid. We are not aware of an example which is tt-rigid but not unitally so.
\end{rem}


\section{Rigidity via recollements}

Now that we have introduced the theory of tt-rigidity, we shall discuss a methodology of proving tt-rigidity for a given $\C$. The idea is as follows: we suppose that we can decompose $\C$ into categories $\C_i$ which can be reassembled to retrieve $\C$. If each $\C_i$ is tt-rigid, then one would like to deduce that $\C$ itself is tt-rigid by showing that the reassembly process is compatible with the tt-rigidity. We will focus our attention to one form of reconstruction here, coming from the theory of \emph{recollements}. 

We begin by discussing the general abstract theory in \cref{sec:abst} before introducing the concept of local duality contexts which will be our main source of recollements in \cref{sec:localdual}. Finally in \cref{sec:meta} we provide the first main theorem of this paper that gives conditions on when we can deduce tt-rigidity from a local duality setup. We will apply this  theorem to a specific example of interest in chromatic homotopy theory in \cref{sec:spectraex}.

\subsection{The general theory}\label{sec:abst}
We recall the pertinent features of recollements here and refer the readers to~\cite[\S A.8]{HA} and~\cite{Shah} for further details.
\begin{defn}
Let $\C$ be a presentable $\infty$-category that admits finite limits and let $\C_0, \C_1 \subset \C$ be full subcategories that are stable under equivalences. Then we say that the pair $(\C_0, \C_1)$ is a \emph{recollement} of $\C$ if the inclusion functors $j_\ast \colon \C_0 \hookrightarrow \C$ and $i_\ast \colon \C_1 \hookrightarrow \C$ admit left exact left adjoints $j^\ast$ and $i^\ast$ such that:
\begin{enumerate}
	\item $j^\ast i_\ast$ is equivalent to the constant functor at the terminal object of $\C_0$;
	\item $j^\ast$ and $i^\ast$ are jointly conservative.
\end{enumerate} 
We shall call $i^\ast j_\ast\colon \C_0 \to \C_1$ the \emph{gluing functor}, the category $\C_0$ the \emph{complete part of the recollement} and $\C_1$ the \emph{local part of the recollement}.
\end{defn}

The next result tells us that, given a recollement, we can reconstruct any object $X \in \C$ via a homotopy pullback of objects in $\C_0$ and $\C_1$.

\begin{prop}[{\cite[Proposition 2.2]{Shah}}]\label{Hasse}
Let $(\C_0,\C_1)$  be a recollement of $\C$. Then there is a pullback square of functors
\[\xymatrix@!0@R=4pc@C=4pc{
\mathrm{id} \ar[r] \ar[d] & i_\ast i^\ast \ar[d] \\
j_\ast j^\ast \ar[r] & i_\ast i^\ast j_\ast j^\ast \rlap{.}
}\]
\end{prop}

Although \cref{Hasse} gives us an objectwise reconstruction, we can in fact elevate this to a categorical reconstruction. This reconstruction will be key to our tt-rigidity machine.

\begin{prop}[{\cite[Corollary 2.12]{Shah}}]\label{prop:recon}
    Let $(\C_0,\C_1)$  be a recollement of $\C$. Then there is a pullback square of presentable $\infty$-categories
\[
\xymatrix@!0@R=4pc@C=4pc{
\C \ar[r]^-{i^\ast \eta_j} \ar[d]_{j^\ast} & \C_1^{\Delta[1]} \ar[d]^{\pi_1} \\
\C_0 \ar[r]_{i^\ast j_\ast} & \C_1
}
\]
where $\eta_j \colon \C \to \C^{\Delta[1]}$ is the functor that sends $X \in \C$ to the unit map $X \to j_\ast j^\ast X$, and $\pi_1$ is the projection to the target.
\end{prop}

\begin{ex}\label{ex:running}
Let us explore a comforting example. To this end, let $\C = \sfD(\Z_{(p)})$ be the derived $\infty$-category of $p$-local abelian groups. Then there is a classical recollement of this situation where $\C_0 = \Lambda_p \sfD(\Z_{(p)})$, the category of derived $p$-complete objects, and $\C_1 = \sfD(\Q)$. The functor $j^\ast$ is derived $p$-completion $\Lambda_p$, while $i^\ast$ is rationalization $\mathbb{Q} \otimes -$.

The objectwise reconstruction of \cref{Hasse} retrieves the usual Hasse square. That is, it tells us that any $X \in \sfD(\Z_{(p)})$ can be recovered as the pullback
\[
\xymatrix@!0@R=4pc@C=4pc{
X \ar[r] \ar[d] &  {\mathbb{Q}} \otimes X \ar[d] \\
\Lambda_p X \ar[r] & {\mathbb{Q}} \otimes \Lambda_p X \rlap{.}
}
\]

The categorical reconstruction of \cref{prop:recon} provides us with the following pullback square
\[
\xymatrix{
\sfD({\Z_{(p)}}) \ar[r] \ar[d] & \sfD(\Q)^{\Delta[1]} \ar[d]^{\pi_1} \\
\Lambda_p\sfD(\Z_{(p)}) \ar[r]_-{\Q \otimes -} & \sfD(\Q) \rlap{.}
}
\]
The natural map is given on objects by the following diagram.
\[
X \longmapsto \left( \begin{gathered}\xymatrix@C=0em{
& (\Q \otimes X \to \Q \otimes \Lambda_p X) \ar@{|->}[d] \\
\Lambda_p X \ar@{|->}[r] & \Q \otimes \Lambda_p X
}\end{gathered} \right)
\]
We warn the reader that the category $\Lambda_p \sfD(\Z_{(p)})$ of derived $p$-complete $\Z_{(p)}$-modules is \emph{not} equivalent to the category $\sfD(\Z_p^\wedge)$. For instance, $\Q \otimes \Z_p^\wedge$ is in the latter, but not the former.
\end{ex}

As we will be interested in questions of rigidity, we will need to have an understanding of morphisms between recollements.

\begin{defn}
 Suppose that $(\C_0,\C_1)$ and $(\D_0,\D_1)$ are recollements of $\C$ and $\D$ respectively. Then a functor $F \colon \C \to \D$ is a \emph{lax morphism of recollements} if $F$ sends $j^\ast_{\C}$-equivalences to  $j^\ast_{\D}$-equivalences and $i^\ast_{\C}$-equivalences to  $i^\ast_{\D}$-equivalences. 
\end{defn}

\begin{rem}\label{remark:morofrec}
The definition of lax morphisms of recollements above is somewhat unsatisfactory. One would like to not have to refer to the ambient categories $\C$ and $\D$. Fortunately this has a remedy~\cite[Observations 2.4 and 2.5]{Shah}. Using the notation from above, the functor $F$ provides functors 
\[F_0 = j_{\D}^\ast F {j_{\C}}_\ast \colon \C_0 \to \D_0 \qquad \text{and} \qquad F_1 = i_{\D}^\ast F {i_{\C}}_\ast \colon \C_1 \to \D_1\]
which assemble into a commutative diagram
\[\xymatrix@!0@R=4pc@C=4pc{
	\C_0 \ar[d]_{F_0} &\ar[l]_{j_{\C}^\ast} \C \ar[r]^{i_{\C}^\ast} \ar[d]|{F}& \C_1 \ar[d]^{F_1} \\
	\D_0 &\ar[l]^{j_{\D}^\ast} \D \ar[r]_{i_{\D}^\ast} & \D_1
}\]
such that $F$ is left exact if and only if $F_0$ and $F_1$ are left exact. From this data one obtains a natural transformation $\alpha \colon F_1 i_{\C}^\ast {j_{\C}}_\ast \Rightarrow i_\D^\ast {j_\D}_\ast F_0$. 
Conversely, if we are given $F_0\colon \C_0 \to \D_0$ and $F_1\colon \C_1 \to \D_1$ together with a natural transformation $\alpha \colon F_1 i_{\C}^\ast {j_{\C}}_\ast \Rightarrow i_\D^\ast {j_\D}_\ast F_0$, then these assemble to give a lax morphism of recollements $F\colon \C \to \D$. 

Explicitly, for $X \in \C$ one defines $FX$ via the following pullback
\begin{equation}\label{defnF}
\begin{gathered}
\xymatrix@!0@R=4pc@C=8pc{
FX \ar[r] \ar[dd] & {i_\D}_\ast F_1 i^\ast_{\C} X   \ar[d]^{{i_\D}_\ast F_1 i^\ast_\C \eta} \\ 
& {i_\D}_\ast F_1 i^\ast_{\C} {j_\C}_\ast j_{\C}^\ast X \ar[d]^{{i_\D}_\ast\alpha_{j^\ast_\C X}} \\
{j_\D}_\ast F_0 j_\C^\ast X \ar[r]_-{\eta} & {i_\D}_\ast i_\D^\ast {j_{\D}}_\ast F_0 j_\C^\ast X \rlap{.}
}
\end{gathered}
\end{equation}
\end{rem}

\begin{defn}
A lax morphism of recollements is \emph{strict} if the natural transformation
\[\alpha \colon F_1 i_{\C}^\ast {j_{\C}}_\ast \Rightarrow i_\D^\ast {j_\D}_\ast F_0
\]
is an equivalence, that is, if the following square commutes.
\[\xymatrix@!0@R=4pc@C=4pc{
\mathcal{C}_0 \ar[r]^{i_{\C}^\ast {j_{\C}}_\ast} \ar[d]_{F_0}& \mathcal{C}_1 \ar[d]^{F_1} \\
\mathcal{D}_0 \ar[r]_{i_\D^\ast {j_\D}_\ast} & \mathcal{D}_1 
}\]
\end{defn}

So far we have discussed the general theory of recollements, but we will be working in a much more specific setup.

\begin{defn}\leavevmode
\begin{itemize}
\item Let $\C$ be a presentable stable $\infty$-category, and let $(\C_0,\C_1)$ be a recollement of $\C$. Then this recollement is \emph{stable} if $\C_0$ and $\C_1$ are stable subcategories.
\item Let $(\C, \otimes, \1)$ be a presentable symmetric monoidal $\infty$-category, and let $(\C_0,\C_1)$ be a recollement of $\C$. Then this recollement is \emph{(symmetric) monoidal} if the functors $j_\ast j^\ast$ and $i_\ast i^\ast$ are compatible with the symmetric monoidal structure of $\C$. That is, for every $j^\ast$-equivalence (resp., $i^\ast$-equivalence) $f \colon X \to Y$ and any $Z \in \C$, $f \otimes \mathrm{id} \colon X \otimes Z \to Y \otimes Z$ is a $j^\ast$-equivalence (resp., $i^\ast$-equivalence). In this situation the gluing functor $i^\ast j_\ast$ is lax symmetric monoidal \cite[Observation 2.21]{Shah}. A lax morphism of monoidal recollements is a lax morphism of recollements such that $F \colon \C \to \D$ is a strong monoidal functor. 
\end{itemize}
\end{defn}

\begin{rem}
When we have a stable recollement $(\C_0, \C_1)$ of a presentable stable $\infty$-category $\C$, the functor $i_*$ has a right adjoint $i^!$, and $j^*$ has a fully faithful left adjoint $j_!$. Therefore we obtain the familiar diagram of functors
\[
\xymatrix@C=20mm{
\C_0 \ar@{<-}[r]|{j^\ast} \ar@<-2ex>@{^(->}[r]_{j_\ast}  \ar@<2ex>@{^(->}[r]^{j_!}  & \C \ar@<2ex>@{->}[r]^{i^\ast} \ar@<0ex>@{<-^)}[r]|{i_\ast} \ar@<-2ex>@{->}[r]_{i^!}  & \C_1  \rlap{.}
} 
\]
\end{rem}

We now record the result that will be of interest to us, of when we can check two stable symmetric monoidal recollements are equivalent.

\begin{thm}\label{thm:equivrec}
    Suppose $\C$ and $\D$ are presentable stable symmetric monoidal $\infty$-categories. Let $(\C_0, \C_1)$ and $(\D_0, \D_1)$ be  stable symmetric monoidal recollements of $\C$ and $\D$ respectively and let $F \colon \C \to \D$ be a strict morphism of recollements. Then $F$ is an equivalence if and only if $F_0 \colon \C_0 \to \D_0$ and $F_1 \colon \C_1 \to \D_1$ are equivalences. Moreover, $F$ is a symmetric monoidal equivalence if and only if $F_0$ and $F_1$ are symmetric monoidal equivalences.
\end{thm}

\begin{proof}
    That $F$ is an equivalence if and only if $F_0$ and $F_1$ are is the subject of \cite[Remark 2.7]{Shah} and \cite[Proposition A.8.14]{HA}. For the claim regarding the monoidality we refer the reader to \cite[Observation 2.32]{Shah} and the discussion following it.
\end{proof}

\begin{rem}
    The strictness of the morphism $F$ in \cref{thm:equivrec} is essential to the proof for the converse direction. For our main application in \cref{sec:spectraex}, proving strictness is the meat of the argument.
\end{rem}


\subsection{Recollements from local duality}\label{sec:localdual}

Now that we have seen the abstract theory of recollements and understood how they can be compared, let us introduce a methodology of producing them. The key point is that in the stable setting, a recollement $(\C_0, \C_1)$ is uniquely determined by the local part $\C_1$~\cite[Proposition A.8.20]{HA}. Indeed, if $\C$ is a presentable stable $\infty$-category with $\C_1$ a stable reflective and coreflective subcategory, then we can define $\C_0$ to be the full subcategory of $\C$ spanned by the objects $X \in \C$ such that $\Hom(Z,X) \simeq 0$ for all $Z \in \C_1$. That is, $\C_0 = \C_1^{\perp}$.

One way of forming stable symmetric monoidal recollements of $\C$ is via the use of \emph{smashing localizations} of $\C$. Recall that a localization $L$ of $\C$ is \emph{smashing} if the natural map $L\1 \otimes X \to LX$ is an equivalence for every $X \in \C$. It is with these ideas in mind that we recall the concept of \emph{local duality contexts} from~\cite{BHV}. First, we fix a hypothesis that all of our examples will satisfy.

\begin{hyp}\label{hyp:running}
$(\C , \otimes , \1)$ will be a presentable stable closed symmetric monoidal $\infty$-category which is compactly generated by dualizable objects.
\end{hyp}

Let $\cK$ be a collection of compact objects in $\C$. We call such a pair $(\C, \cK)$ a \emph{local duality context}. We write $\Gamma_\cK\C$ for the localizing tensor-ideal in $\C$ generated by $\cK$ and define $L_\cK \C = (\Gamma_\K\C)^\perp$ and $\Lambda_\cK \C = (L_\K\C)^\perp$. When no confusion is likely to occur we will drop the subscript $\K$ from the notation. We note that these subcategories do not depend on the precise choice of compact objects $\K$, but rather on the thick tensor-ideal which they generate. There are corresponding inclusion functors
\[
\iota_{\Gamma} \colon \Gamma \C \hookrightarrow \C, \quad
\iota_{L} \colon L\C \hookrightarrow \C, \quad
\iota_{\Lambda} \colon \Lambda \C \hookrightarrow \C.
\]

Let us recall some of the salient features of the above formalism from~\cite{axiomatic, GreenleesTate, BHV}. In particular, \cref{prop:localduality}\ref{localduality6} tells us that local duality contexts allow us to obtain recollements.

\begin{prop}\label{prop:localduality}
Let $\C$ satisfy \cref{hyp:running}, and $\cK$ be a set of compact objects in $\C$.
\begin{enumerate}[label=(\arabic*)]
	\item The functors $\iota_\Gamma$ and $\iota_L$ have right adjoints denoted by $\Gamma$ and $V$ respectively, and the functors $\iota_L$ and $\iota_\Lambda$ have left adjoints denoted $L$ and $\Lambda$ respectively. These induce natural cofibre sequences
	 \[\Gamma X \to X \to LX \qquad \text{and} \qquad
    VX \to X \to \Lambda X\]
for all $X \in \C$.
\item The (co)localizations $\Gamma$ and $L$ are smashing.
\item\label{localduality3} The functors $\Lambda \iota_\Gamma \colon \Gamma \C \to \Lambda \C$ and $\Gamma \iota_\Lambda \colon \Lambda \C \to \Gamma \C$ are mutually inverse equivalences of stable $\infty$-categories. Moreover there are natural equivalences of functors
\[\Lambda \Gamma \xrightarrow{\sim} \Lambda \qquad \Gamma \xrightarrow{\sim} \Gamma \Lambda.\]
\item When viewed as endofunctors on $\C$ via the inclusions, the functors $(\Gamma, \Lambda)$ form an adjoint pair in that there is a natural equivalence
    \[\Hom_\C^\C(\Gamma X, Y) \simeq \Hom_\C^\C(X, \Lambda Y)\]
    for all $X,Y \in \C$. In particular we have $\Hom_\C^\C(\Gamma \unit, Y) \simeq \Lambda Y$.
    \item For every $X \in \C$  there is a pullback square
    \[
    \xymatrix@!0@R=4pc@C=4pc{X \ar[r] \ar[d] & LX \ar[d] \\ \Lambda X \ar[r] & L \Lambda X \rlap{.}}
    \]
    whose vertical and horizontal fibres are $VX$ and $\Gamma X$ respectively. \label{localduality5}
    \item The pair $(\Lambda \C , L \C)$ is a stable symmetric monoidal recollement of $\C$ where the gluing data is described by the inclusion followed by the localization $L$.\label{localduality6}
\end{enumerate}
\end{prop}

 

\begin{ex}
We note that the recollement from \cref{ex:running} may be obtained from the local duality context $(\msf{D}(\mathbb{Z}_{(p)}), \{\mathbb{Z}_{(p)}/p\})$.
\end{ex}

We end this section with an observation regarding how local duality interacts with tt-equivalences. There are more general statements of this form in the literature where the functor need not be an equivalence, but this is not required for us. We refer the interested reader to~\cite[Proposition 2.7]{Stevenson} for the more general statement.

\begin{lem}\label{restricts}
Let $\K$ be a set of compact objects of $\C$ and let $\Phi\colon h\C \to h\D$ be a tt-equivalence. Then there are equivalences
\begin{enumerate}[label=(\roman*)]
\item $\Phi \Gamma_\K X \simeq \Gamma_{\Phi \K} \Phi X$ for every $X \in \C$,\label{l1:i}
\item $\Phi L_\K X \simeq L_{\Phi \K} \Phi X$ for every $X \in \C$,\label{l1:ii}
\item $\Phi\Lambda_\K X  \simeq \Lambda_{\Phi \K} \Phi X$ for every $X \in \C$.\label{l1:iii}
\end{enumerate}
\end{lem}

\begin{proof}
    For \ref{l1:i} it suffices to prove that $\Phi \Gamma_{\K} X$ is in the localizing tensor-ideal $\mathrm{Loc}^{\otimes}(\Phi \K)$ and that the natural map $\Phi \Gamma_\K X \to \Phi X$ is a $(\Phi \K \otimes \G)$-cellular equivalence where $\G$ is a set of compact generators for $\D$. As $\Phi$ is an equivalence these are clear. Parts \ref{l1:ii} and \ref{l1:iii} follow similarly. 
\end{proof}

\subsection{The strategy}\label{sec:meta}

We are now ready to assemble the strategy that we will use to deduce tt-rigidity via the theory of recollements.

We fix $\C$ satisfying \cref{hyp:running}. We also fix a set of compact objects $\K$ to form a local duality context $(\C,\K)$, and write $L$ and $\Lambda$ for the associated localization and completion functors. 
Suppose we are given a tt-equivalence $\Phi\colon h\C \to h\D$, for $\D$ a presentable stable closed symmetric monoidal $\infty$-category; as such $\D$ also satisfies \cref{hyp:running} since the compact generation can be verified at the homotopy level~\cite[Remark 1.4.4.3]{HA}. 
For brevity, we then write $L'$ and $\Lambda'$ for the localization and completion associated to the local duality context $(\D, \Phi(\K))$ passed along the equivalence $\Phi$.

The category of complete objects $\Lambda\C$ inherits a closed symmetric monoidal structure, with monoidal product given by the \emph{completed tensor product} $- \whotimes - := \Lambda(- \otimes -)$, internal hom the same as in the underlying category, and tensor unit $\Lambda\1_{\C}$. On the other hand, since $L$ is a smashing localization, in $L\C$ the tensor product is the same as in the underlying category, and the tensor unit is $L\1_{\C}$.

\begin{lem}\label{ttrestricts}
Let $\Phi\colon h\C \xrightarrow{\sim} h\D$ be a tensor-triangulated equivalence, and let $\K$ be a set of compact objects in $\C$. Then $\Phi$ restricts to give tt-equivalences $hL\C \xrightarrow{\sim} hL'\D$ and $h\Lambda\C \xrightarrow{\sim} h\Lambda'\D$. 
\end{lem}
\begin{proof}
That $\Phi$ restricts to a triangulated equivalence on the relevant categories is immediate from \cref{restricts}, so it remains to check that the restrictions are monoidal. As $L$ is smashing, this is clear for the local case, so it suffices to check for $\Lambda$. There are equivalences
\[\Phi(X \whotimes Y) = \Phi(\Lambda(X \otimes Y)) \simeq \Lambda'(\Phi(X \otimes Y)) \simeq \Lambda'(\Phi X \otimes \Phi Y) = \Phi X \whotimes \Phi Y \] from which the result follows.
\end{proof}

We now suppose that $L\C$ and $\Lambda\C$ are tt-rigid. Therefore by \cref{ttrestricts} we have equivalences $F_L\colon L\C \xrightarrow{\sim} L'\D$ and $F_\Lambda\colon \Lambda\C \xrightarrow{\sim} \Lambda'\D$ on the level of $\infty$-categories. 
We consider the following diagram
\begin{equation}\label{diagram}
\begin{gathered}\xymatrix{
\Lambda \C \ar[d]_{F_\Lambda}^{\simeq} \ar[r]^{L}& L\C \ar[d]^{F_L}_{\simeq} \\ 
\Lambda'\D \ar[r]_{L'}& L' \D \rlap{.}
}
\end{gathered}
\end{equation}
If the diagram (\ref{diagram}) commutes, then the induced functor $F\colon \C \to \D$ of \cref{remark:morofrec} is an equivalence by \cref{thm:equivrec}. 
In particular, $\C$ is tt-rigid. 
Now if $L\C$ and $\Lambda\C$ were \emph{unitally} tt-rigid, then so is $\C$. Indeed, by (\ref{defnF}), one sees that $F(\1_\C)$ can be described as the pullback 
\[\xymatrix{
F(\1_\C) \ar[r] \ar[d] & F_L(L\1_\C) \ar[d] \\
F_\Lambda(\Lambda\1_\C) \ar[r] & L'F_\Lambda(\Lambda\1_\C)
\rlap{.} } \]
As $F_L$ and $F_\Lambda$ preserve the respective tensor units, this pullback coincides with the decomposition of $\1_\D$ as in \cref{Hasse} (also see \cref{prop:localduality}\ref{localduality5}). 
Thus, $F(\1_\C) \simeq \1_\D$ and $\C$ is unitally tt-rigid.
Finally, if $L\C$ and $\Lambda\C$ were \emph{strongly} tt-rigid, then $\C$ is strongly tt-rigid by \cref{thm:equivrec}. In summary, we have proved the following strategy for proving tt-rigidity.

\begin{thm}\label{metatheorem}
Let $\C$ satisfy \cref{hyp:running}. Suppose that for any tt-equivalence $\Phi \colon h \C \to h \D$ there is a local duality context $(\C,\K)$ with associated localization $L$ and completion $\Lambda$ such that
\begin{enumerate}
    \item $L \C$ and $\Lambda \C$ are tt-rigid (resp., unitally tt-rigid, resp., strongly tt-rigid),
    \item the induced diagram (\ref{diagram}) commutes.
\end{enumerate}
Then $\C$ is tt-rigid (resp., unitally tt-rigid, resp., strongly tt-rigid).
\end{thm}


\begin{rem}
Although \cref{metatheorem} asks for the choice of a local duality context for each equivalence $\Phi \colon h \C \to h \D$, in practice, there will usually be a single local duality context that covers all cases as we will see in \cref{sec:spectraex}. Moreover, with a tensor-triangular view on the situation, the choice of local duality context is often suggested by the philosophy of considering the local duality context provided by the closed points of a Noetherian Balmer spectrum.
\end{rem}

\begin{rem}
We have focused on recollements as the reconstruction technique here, but other reconstructions are available in the literature. For example, there is the adelic module model~\cite{adelicm} which provides a more algebraic reconstruction of the category at hand, replacing the category of complete modules $\Lambda\C$ with the category of modules over the completed unit, $\bmod{\Lambda\1}$. However in this setting, the analogue of \cref{ttrestricts} is no longer evident and we expect it to be largely dependent on the particular example at hand. One particular example where this model has proved a powerful tool is in rational equivariant stable homotopy theory~\cite{GStorus}. For rational torus-equivariant spectra of arbitrary rank $r$, the adelic module model provides a diagram indexed on a punctured $(r-1)$-cube, each of whose vertices is strongly tt-rigid via \cref{formalhomotopy}. So the remaining obstruction to applying \cref{metatheorem} in this case, is proving an analogue of \cref{ttrestricts}.
\end{rem}

\section{Rigidity in chromatic homotopy theory}\label{sec:spectraex}

In this section we will apply \cref{metatheorem} to the case $\C = L_n \Sp_{(p)}$ and use it to prove that the tt-rigidity of the $E(n)$-local category is controlled by the $K(i)$-local categories for $i \leqslant n$. 

We begin in \cref{subsec:specialfunctors} discussing properties of left adjoints out of $\Sp$ which will be an essential ingredient in our main proof. It moreover allows us to prove that localizations of spectra are unitally tt-rigid if and only if they are strongly tt-rigid, as well as giving a criterion for proving unital tt-rigidity from rigidity. We then provide the proof of the aforementioned rigidity result in \cref{subsec:mainres}.

\subsection{The universal property of spectra and consequences}\label{subsec:specialfunctors}

Let $\D$ be a presentable stable $\infty$-category. Then in particular $\D$ is enriched and tensored over spectra~\cite[Proposition 4.8.2.18]{HA}. That is, for any two objects $A,B \in \D$ we have a mapping spectrum $\Hom^{\Sp}_{\D}(A,B) \in \Sp$, and for any $X \in \Sp$ and $A \in \D$, we have $A \odot X \in \D$ satisfying the usual enriched adjunction
\[\Hom_\D^\Sp(A \odot X, B) \simeq \Hom_\Sp^\Sp(X, \Hom_\D^\Sp(A,B))\] for all $A, B \in \D$ and $X \in \Sp$. We also recall that there is a natural equivalence
\begin{equation}\label{tensoroftensor}
A \odot (X \otimes Y) \simeq (A \odot X) \odot Y
\end{equation} for all $A \in \D$ and $X,Y \in \Sp$, by a standard adjunction argument. For clarity, we emphasize that we use $\odot$ for the enriched tensor, and $\otimes$ for the monoidal tensor. 

\begin{prop}\label{tensors}
Let $\D$ be a presentable stable $\infty$-category. 
\begin{enumerate}[label=(\roman*)]
\item Evaluation at the sphere spectrum $S^0$ yields an equivalence of $\infty$-categories \[\mathrm{Fun}^\mathrm{L}(\Sp,\D) \xrightarrow{\sim} \D\] where $\mathrm{Fun}^\mathrm{L}(-,-)$ denotes the $\infty$-category of colimit-preserving functors.
\item\label{tensor:2} Any colimit-preserving functor $F \colon \Sp \to \D$ is of the form
\[
F(S^0) \odot - \colon \Sp \to \D.
\]
\item\label{tensor:3} Moreover, a local version also holds: for any localization $L$ of spectra, any colimit-preserving functor
$F \colon L \Sp \to \D$ is of the form
\[F(LS^0) \odot -\colon L\Sp \to \D.\]
\end{enumerate}
\end{prop}
\begin{proof}
The equivalence of $\infty$-categories $\mathrm{Fun}^\mathrm{L}(\Sp,\D) \xrightarrow{\simeq} \D$ is the universal property of spectra~\cite[Corollary 1.4.4.6]{HA}. A quasi-inverse to this equivalence is given by $A \mapsto A \odot -$ from which the identification of colimit-preserving functors follows. The local version follows from the general case applied to the colimit-preserving composite $\Sp \xrightarrow{L} L\Sp \xrightarrow{F} \D$.
\end{proof}

\begin{lem}\label{tensorandtensor}
Let $\D$ be a presentable stable closed symmetric monoidal $\infty$-category. For any $A, B \in \D$, the functors $(A \otimes B) \odot -\colon \Sp \to \D$ and $A \otimes (B \odot -)\colon \Sp \to \D$ are naturally equivalent.
\end{lem}
\begin{proof}
The functor $G = A \otimes (B \odot -)\colon \Sp \to \D$ is colimit-preserving, and so is of the form $G(S^0) \odot -$ by \cref{tensors}\ref{tensor:2}. We have $G(S^0) \simeq A \otimes B$ which gives the claim.
\end{proof}

As a consequence we have the following, 
which may also be found (in a slightly different setting) as~\cite[Theorem 6.4]{stable_frames} and~\cite[Corollary 4.8.2.19]{HA}. 
\begin{cor}\label{tensoringismonoidal}
Let $\D$ be a presentable stable closed symmetric monoidal $\infty$-category. Then the functor $\1_\D \odot -\colon \Sp \to \D$ is symmetric monoidal.
\end{cor}
\begin{proof}
Firstly, note that $\1_\D \odot S^0 \simeq \1_\D$ so that the functor preserves the unit. The right adjoint of $\1_\D \odot -$ is lax symmetric monoidal, so $\1_\D \odot -$ is oplax symmetric monoidal. Therefore it suffices to check that $\1_\D \odot (X \otimes Y) \to (\1_\D \odot X) \otimes (\1_\D \odot Y)$ is an equivalence for all $X, Y \in \Sp$. 
Consider the functor \[G = (\1_\D \odot X) \otimes (\1_\D \odot -)\colon \Sp \to \D.\]
This is colimit-preserving, and so by \cref{tensors}\ref{tensor:2}, we have $G \simeq G(S^0) \odot -$.  
We have $G(S^0) \simeq \1_\D \odot X$, and therefore \[G = (\1_\D \odot X) \otimes (\1_\D \odot -) \simeq (\1_\D \odot X) \odot -\] which in turn is equivalent to $\1_\D \odot (X \otimes -)$ by (\ref{tensoroftensor}) as required.
\end{proof}



Recall that for any localization $L$ of spectra, the category $L\Sp$ of $L$-local spectra inherits a closed symmetric monoidal structure with monoidal product given by the localized tensor product $- \whotimes - := L(- \otimes -)$, and monoidal unit $LS^0$. When $L$ is a smashing localization, the localized tensor product agrees with the underlying tensor product of spectra.

For $\D$ a presentable stable $\infty$-category there is an enriched tensor $A \odot - \colon \Sp \to \D$ for any $A \in \D$ as recalled above. For any localization $L$ of $\Sp$, this provides us with a functor $A \odot - \colon L \Sp \to \D$ via restriction.

\begin{lem}\label{localhoms}
Let $L$ be any localization of spectra, and let $\D$ be a presentable stable closed symmetric monoidal $\infty$-category. 
If $\1_\D \odot -\colon L\Sp \to \D$ is an equivalence, then $\Hom_\D^\Sp(\1_\D, A)$ is $L$-local for any $A \in \D$.
\end{lem}
\begin{proof}
    Since $\1_\D \odot -$ is an equivalence, we have objects $X,Y \in L\Sp$ such that $\1_\D \odot X \simeq \1_\D$ and $\1_\D \odot Y \simeq A$. Therefore
    \[\Hom_\D^\Sp(\1_\D,A) \simeq \Hom_\D^\Sp(\1_\D \odot X, \1_\D \odot Y) \simeq \Hom_\Sp^\Sp(X,Y).\] 
    As $Y$ is $L$-local, $\Hom_\Sp^\Sp(X,Y)$ is $L$-local. Indeed, to prove this it suffices to show that if $LZ \simeq 0$, then \[\Hom_\Sp^\Sp(Z,\Hom_\Sp^\Sp(X,Y)) \simeq 0.\] If $LZ \simeq 0$, then $L(X \otimes Z) \simeq 0$ by definition of a localization, and hence by adjunction the claim follows.
\end{proof}

\begin{prop}\label{rigidimpliesunitally}
Let $L$ be any localization of spectra, and let $\D$ be a presentable stable closed symmetric monoidal $\infty$-category. If $\1_\D \odot -\colon L\Sp \to \D$ is an equivalence, then $\1_\D \odot LS^0 \simeq \1_\D$.
\end{prop}
\begin{proof}
Recall that $\1_\D \odot S^0 \simeq \1_\D$. So we argue that the natural map $\1_\D \odot S^0 \to \1_\D \odot LS^0$ is an equivalence. As $\1_\D \odot -\colon L\Sp \to \D$ is an equivalence by assumption, $\Hom_\D^\Sp(\1_\D,A)$ is $L$-local for all $A \in L\D$ by \cref{localhoms}. As such, by adjunction we have equivalences
\begin{align*}
    \Hom_\D^\Sp(\1_\D \odot LS^0, A) &\simeq \Hom_\Sp^\Sp(LS^0, \Hom_\D^\Sp(\1_\D, A)) \\& \simeq \Hom_\Sp^\Sp(S^0, \Hom_\D^\Sp(\1_\D, A)) \\& \simeq \Hom_\D^\Sp(\1_\D \odot S^0, A)
\end{align*}
so the claim follows.
\end{proof}

In \cref{ex:unitally}, we recalled that the category $\Sp$, along with some of its localizations are known to be rigid. In the proofs of these results, the functor realising the equivalence between the $\infty$-categories is of the form given in \cref{rigidimpliesunitally}. As such, we conclude that these examples are not just rigid, but in fact unitally tt-rigid.

What is more, we now show that unital tt-rigidity is equivalent to strong tt-rigidity for localizations of spectra. This is a generalization of a theorem of Shipley~\cite[Theorem 4.7]{ShipleyMonoidal}, using the language of tt-rigidity. 
\begin{prop}\label{unitallyimpliesstrongly}
Let $L$ be a localization of spectra. Then $L\Sp$ is strongly tt-rigid if and only if it is unitally tt-rigid.
\end{prop}
\begin{proof}
Suppose that $L\Sp$ is unitally tt-rigid. Any equivalence $F\colon L\Sp \to \D$ is of the form $F(LS^0) \odot -$ by \cref{tensors}\ref{tensor:3}. 
By the unitality assumption, we know that $F(LS^0) \simeq \1_\D$, so $F \simeq \1_\D \odot -$. 
This is symmetric monoidal as a functor $\Sp \to \D$ by \cref{tensoringismonoidal}, but we must check that its restriction to a functor $L\Sp \to \D$ is symmetric monoidal. 
In order to do this, by the definition of the monoidal product in $L\Sp$, it suffices to check that the natural map \[\1_\D \odot (X \otimes Y) \to \1_\D \odot L(X \otimes Y)\] is an equivalence. This follows from \cref{localhoms} in a similar way to the proof of \cref{rigidimpliesunitally}.
Therefore $F\colon L\Sp \to \D$ is a symmetric monoidal equivalence, and hence $L\Sp$ is strongly tt-rigid.
\end{proof}

With the previous two results in hand, we may now return to \cref{ex:unitally}.
\begin{ex}\label{ex:spectra}
    The categories $\Sp$, $L_1\Sp_{(2)}$ and $L_{K(1)}\Sp_{(2)}$ are all strongly tt-rigid. 
    Write $\C$ to denote any of $\Sp$, $L_1\Sp_{(2)}$ or $L_{K(1)}\Sp_{(2)}$. 
    We suppose we have a tt-equivalence $\Phi\colon h\C \to h\D$. 
    In each of the 3 cases, the functor $\Phi(\1_\C) \odot -$ is known to be an equivalence, see~\cite{Schwederigidity, Roitzheimrigidity, Ishakrigidity} respectively. 
    As $\Phi$ is a tt-functor, $\Phi(\1_\C) \simeq \1_\D$. 
    Therefore, by \cref{rigidimpliesunitally} we see that $\Sp$, $L_1\Sp_{(2)}$ and $L_{K(1)}\Sp_{(2)}$ are unitally tt-rigid. 
    Applying \cref{unitallyimpliesstrongly} shows that they are all moreover strongly tt-rigid.

\end{ex}

We finish this section with an auxiliary lemma which will be needed in \cref{subsec:mainres}.

\begin{lem}\label{enrichedlocaltensor}
Let $\D$ be a  presentable stable closed symmetric monoidal $\infty$-category and $L$ be any localization of $\D$. Write $\odot$ for the enriched tensor of $\D$ over $\Sp$.
\begin{enumerate}[label=(\roman*)]
\item The enriched tensor of $L\D$ over $\Sp$ is given by $- \whodot - := L(- \odot -)$.\label{l3:i}
\item If $L$ is smashing, the enriched tensor of $L\D$ over $\Sp$ is $\odot$.\label{l3:ii}
\end{enumerate}
\end{lem}
\begin{proof}
Part \ref{l3:i} follows from the defining universal property. For part \ref{l3:ii}, if $A$ is $L$-local, \[A \odot X \simeq (L\1_\D \otimes A) \odot X \simeq L\1_\D \otimes (A \odot X) \simeq A \whodot X\] as required, where the second equivalence follows from \cref{tensorandtensor}.
\end{proof}

\subsection{Tensor-triangular rigidity in chromatic homotopy theory}\label{subsec:mainres}

Throughout the rest of this paper, we work $p$-locally and suppress this from the notation.  In this section we will prove our main results regarding tt-rigidity in chromatic homotopy theory. Our key result is the following.
\begin{thm}\label{main}
Let $n\geqslant 1$. If $L_{n-1}\Sp$ and $L_{K(n)}\Sp$ are unitally tt-rigid, then $L_n\Sp$ is strongly tt-rigid.
\end{thm}

Before giving the proof, let us first record the following corollary of this theorem which follows by a simple inductive argument and the observation that $L_0 \Sp \simeq \bmod{H\Q}$ is strongly tt-rigid (e.g., by \cref{formalhomotopy}).
\begin{cor}\label{K(n)}
Let $n \geqslant 1$. If $L_{K(i)}\Sp$ is unitally tt-rigid for all $1 \leqslant i \leqslant n$, then $L_n\Sp$ is strongly tt-rigid. \qed
\end{cor}

 We now turn to providing a proof of \cref{main}. We note that the category $\Sp$ and its localizations $L_n \Sp$ satisfy \cref{hyp:running}, so that we can implement the strategy devised in \cref{metatheorem}.  Since the proof requires several lemmas and steps, we begin by fixing the setup and describing the key elements of the proof.
\begin{strat}\label{strat}
Suppose that there is a tt-equivalence $\Phi\colon hL_n\Sp \to h\D$ where $\D$ is a  presentable stable closed symmetric monoidal $\infty$-category. To apply \cref{metatheorem}, we have three key steps.
\begin{enumerate}
    \item Pick a suitable local duality context on $L_n \Sp$.
    \item Prove that there is a lax morphism of recollements between the chosen local duality context and the one on $\D$ induced by the equivalence.
    \item Prove that this lax morphism is actually strict.
\end{enumerate}
We will address each of these in turn, which will then assemble to give a proof of \cref{main}.
\end{strat}

Henceforth we assume the setup of \cref{strat}. We consider the local duality context $(L_n\Sp, F(n))$ where $F(n)$ is a finite type $n$ complex. 
The associated localization functor is $L_{n-1}$ and the associated completion is $L_{K(n)}$, and this yields the recollement $(L_{K(n)}\Sp, L_{n-1}\Sp)$ of $L_n\Sp$ \cite[\S 6]{BHV}. 
We consider the corresponding local duality context $(\D, \Phi(F(n)))$ and write $\Gamma$, $L$ and $\Lambda$ for the associated torsion, localization, and completion functors.

By \cref{restricts}, the tt-equivalence $\Phi$ restricts to the local and complete parts.  By assumption, $L_{n-1}\Sp$ and $L_{K(n)}\Sp$ are unitally tt-rigid, and as such by \cref{unitallyimpliesstrongly} are strongly tt-rigid. Therefore we obtain symmetric monoidal equivalences \[F_L\colon L_{n-1}\Sp \xrightarrow{\simeq} L\D \quad \text{and} \quad F_\Lambda\colon L_{K(n)}\Sp \xrightarrow{\simeq} \Lambda\D.\] 

By \cref{tensors}\ref{tensor:3}, we have $F_L \simeq F_L(L_{n-1}S^0) \odot - \simeq L\1_\D \odot -$ and similarly $F_\Lambda \simeq \Lambda\1_\D \whodot -$ where we implicitly used \cref{enrichedlocaltensor} to identify the enriched tensors of $L\D$ and $\Lambda\D$ over spectra. Since $F_L$ and $F_\Lambda$ are symmetric monoidal equivalences, in particular we have equivalences
\begin{equation}\label{units}
    \rho \colon L\1_\D \odot L_{n-1}S^0 \xrightarrow{\simeq} L\1_\D \quad \text{and} \quad \xi \colon \Lambda\1_\D \whodot L_{K(n)}S^0 \xrightarrow{\simeq} \Lambda\1_\D.
\end{equation}

All of this discussion leads to the following square
\begin{equation}\label{firstsquare}
\begin{gathered}\xymatrix{
    L_{K(n)} \Sp \ar[r]^{L_{n-1}} \ar[d]_{\Lambda \1_\D \whodot -}^{\simeq} & L_{n-1} \Sp \ar[d]_{\simeq}^{L \1_\D\odot -} \\
    \Lambda \D \ar[r]_{L} & L \D \rlap{.}
}\end{gathered}
\end{equation}

\begin{lem}\label{lax}
    There is a natural transformation $\alpha\colon L\1_\D \odot L_{n-1}(-) \Rightarrow L(\Lambda\1_\D \whodot -)$ which provides a lax morphism of recollements $L_n\Sp \to \D$.
\end{lem}

\begin{proof}
    We need to show that there is a natural transformation $\alpha\colon L\1_\D \odot L_{n-1}(-) \Rightarrow L(\Lambda\1_\D \whodot -)$; that is, a natural map between the two paths around the square (\ref{firstsquare}). By taking the adjoints of the vertical equivalences in (\ref{firstsquare}) we obtain the second square
    \begin{equation}\label{secondsquare}
    \begin{gathered}\xymatrix{
    L_{K(n)} \Sp \ar[r]^{L_{n-1}} \ar@{<-}[d]_{\Hom_\D^\Sp(\1_\D,-)}^{\simeq} & L_{n-1} \Sp \ar@{<-}[d]_{\simeq}^{\Hom_\D^\Sp(\1_\D,-)} \\
    \Lambda \D \ar[r]_{L} & L \D \rlap{.}
}
\end{gathered}
\end{equation}
We note that the adjoints of the vertical equivalences in (\ref{firstsquare}) are given by $\Hom_\D^\Sp(\Lambda\1_\D, -)$ and $\Hom_\D^\Sp(L\1_\D,-)$ respectively by \cref{localhoms}, and that by adjunction these are moreover equivalent to $\Hom_\D^\Sp(\1_\D,-)$ on their respective domains. 

The data of $\alpha$ is equivalent to the data of a natural transformation \[\widetilde{\alpha}\colon L_{n-1}\Hom_\D^\Sp(\1_\D,-) \Rightarrow \Hom_\D^\Sp(\1_\D,L(-))\] of functors $\Lambda\D \to L_{n-1}\Sp$ between the two paths around   (\ref{secondsquare}). More explicitly, given $\widetilde{\alpha}$ we obtain $\alpha$ as the composite
\begin{align*}
L\1_\D \odot L_{n-1}(-) \xRightarrow{\mathmakebox[65pt]
{L\1_{\D} \odot L_{n-1}\mathrm{coev}_{\Lambda}}} & L\1_\D \odot L_{n-1}\Hom_\D^\Sp(\1_\D, \Lambda\1_\D \whodot -) \\
\xRightarrow{\mathmakebox[65pt]{L\1_{\D} \odot \tilde{\alpha}}} & L\1_\D \odot \Hom_\D^\Sp(\1_\D, L(\Lambda\1_\D \whodot -)) \\
\xRightarrow{\mathmakebox[65pt]{\mathrm{ev}_{L}}} & L(\Lambda\1_\D \whodot -),
\end{align*}
where $\mathrm{coev}_{\Lambda}$ is the unit of the adjunction $(\Lambda\1_{\D} \whodot-,\Hom_{\D}^{\Sp}(\1_{\D},-))$, and $\mathrm{ev}_L$ is the counit of the corresponding adjunction for $L$.

To construct such an $\widetilde{\alpha}$ we note that the natural map $\Hom_\D^\Sp(\1_\D, -) \Rightarrow \Hom_\D^\Sp(\1_\D, L(-))$ factors over $L_{n-1}\Hom_\D^\Sp(\1_\D, -)$ as $\Hom_\D^\Sp(\1_\D, L(-))$ is $L_{n-1}$-local by \cref{localhoms}. Therefore we have a natural transformation $\alpha$ as required.
\end{proof}

As we now have a lax morphism of recollements, we can apply the discussion of \cref{remark:morofrec} to obtain a functor $F \colon L_n \Sp \to \D$. The following lemmas allows us to deduce essential facts about $F$.

\begin{lem}\label{lem:Funit}
We have $F(L_nS^0) \simeq \1_\D$.
\end{lem}
\begin{proof}
Consider the diagram
\[\xymatrix{
F(L_nS^0) \ar[rr] \ar[dd] & & L\1_\D \odot L_{n-1}S^0 \ar[dr]_{\simeq}^{\rho} \ar[dd]|\hole_<<<<<<{\alpha \circ (L\1_\D \odot L_{n-1}\eta_\Lambda)} & \\
& \1_\D \ar[rr] \ar[dd] & & L\1_\D \ar[dd]^{L\eta_\Lambda} \\
\Lambda\1_\D \whodot L_{K(n)}S^0 \ar[dr]_{\simeq}^{\xi} \ar[rr]_<<<<<<<{\eta_L}|-\hole & & L(\Lambda\1_\D \whodot L_{K(n)}S^0) \ar[dr]_{\simeq}^{L\xi} & \\
& \Lambda\1_\D \ar[rr]_{\eta_L} & & L\Lambda\1_\D 
}\]
Both the back face and the front face of this diagram are pullbacks by (\ref{defnF}) and \cref{prop:localduality}\ref{localduality5} respectively. The bottom square of the diagram commutes by naturality of $\eta_L$, so if the right hand face commutes, then we obtain an induced map $F(L_nS^0) \to \1_\D$ which is an equivalence as required. We verify that the right hand face commutes in \cref{bigdiagram}, which completes the proof.
\end{proof}

\begin{lem}\label{lem:itsmonoidal}
    The functor $F$ is naturally equivalent to the functor $\1_{\D} \odot -\colon L_n \Sp \to \D$. Therefore it is a coproduct-preserving and compact-preserving symmetric monoidal functor.
\end{lem}
\begin{proof}
   By the universal property of spectra (c.f., \cref{tensors}) together with \cref{lem:Funit}, we have $F \simeq \1_{\D} \odot -$. 
    It is clear that this functor is coproduct-preserving. Note that $\1_\D$ is compact since we have a tt-equivalence $\Phi\colon hL_n\Sp \to h\D$ which sends the compact unit $L_nS^0$ of $hL_n\Sp$ to $\1_\D$. Therefore the functor $\1_\D \odot -$ preserves compacts, since its right adjoint $\Hom_\D^\Sp(\1_\D,-)$ preserves sums as $\1_\D$ is compact.  Finally, it is also monoidal. Indeed, we have shown that it preserves the unit, and \cref{tensoringismonoidal} shows it preserves the tensor product as $L_{n}$ is smashing and as such the tensor product of $L_{n}\Sp$ coincides with the tensor product in $\Sp$. 
\end{proof}

For $X \in L_n\Sp$, we write $M_n X$ for the fibre of the natural localization map $X \to L_{n-1} X$, and note that this is the torsion functor arising from the local duality context $(L_n\Sp, F(n)).$

\begin{lem}\label{colocals}
For any $X \in L_n\Sp$, we have $\1_\D \odot M_nX \simeq \Gamma(\1_\D \odot M_nX).$
\end{lem}
    \begin{proof}
        We have a cofibre sequence
    $L \1_\D \odot M_n X \to L\1_\D \odot X \to L \1_\D \odot L_{n-1} X.$
        We then identify \[L\1_\D \odot L_{n-1} X \simeq L \1_\D \odot (L_{n-1}S^0 \otimes X) \simeq (L \1_\D \odot L_{n-1} S^0) \odot X \simeq L \1_\D \odot X\] by using that $L_{n-1}$ is smashing, (\ref{tensoroftensor}) and (\ref{units}) in turn. 
        As such, we have $L \1_\D \odot M_n X \simeq 0$ and therefore by considering the cofibre sequence 
        \[\Gamma\1_\D \odot M_nX \to \1_\D \odot M_nX \to L\1_\D \odot M_nX\] we see that $\Gamma\1_\D \odot M_nX \simeq \1_\D \odot M_nX$.
        The claim then follows from \cref{tensorandtensor} as $\Gamma$ is smashing.
    \end{proof}


\begin{lem}\label{equalthicks}
    There is an equality $\mrm{Thick}(\Phi (F(n))) = \mrm{Thick}(\1_\D \odot F(n))$.
\end{lem}
\begin{proof}
    By the $K(n)$-local thick subcategory theorem~\cite[Proposition 12.1]{HS}, $\mrm{Thick}(X) = \mrm{Thick}(Y)$ for all non-zero $X,Y \in (L_{K(n)}\Sp)^\omega$ where $(-)^\omega$ denotes the full subcategory of compact objects. 
    Since we have an equivalence of categories $\Gamma\D \simeq L_{K(n)}\Sp$ by composing the equivalence of \cref{prop:localduality}\ref{localduality3} with the equivalence $F_\Lambda$, it follows that $\mrm{Thick}(A) = \mrm{Thick}(B)$ for all non-zero $A,B \in (\Gamma\D)^\omega$. 
    As such, it suffices to show that $\Phi (F(n))$ and $\1_\D \odot F(n)$ are in $(\Gamma\D)^\omega$.  

    We observe that $\Gamma \D \cap \D^\omega \subseteq (\Gamma \D)^\omega$. As $F(n)$ is compact in $L_n \Sp$, both $\Phi (F(n))$ and $\1_\D \odot F(n)$ are compact in $\D$ (using \cref{lem:itsmonoidal} for the latter). We have $\Phi (F(n)) \in \Gamma\D$ by \cref{restricts}, and $\1_\D \odot F(n) \in \Gamma\D$ by \cref{colocals}, hence $\Phi(F(n))$ and $\1_D \odot F(n)$ are both in $(\Gamma\D)^\omega$ and therefore generate the same thick subcategory.
\end{proof}

\begin{lem}\label{lem:iffnthengood}
There is a natural equivalence $\1_\D \odot L_{n-1}(-) \simeq L(\1_\D \odot -)$.
\end{lem}

\begin{proof}
    By \cref{lem:itsmonoidal} we can invoke \cite[Proposition 2.7]{Stevenson}, which tells us that the localization $L'$ associated to the local duality context $(\D, \1_\D \odot F(n))$ satisfies $\1_\D \odot L_{n-1}(-) \simeq L'(\1_\D \odot -)$. By \cref{equalthicks}, the local duality contexts $(\D, \Phi (F(n)))$ and $(\D, \1_\D \odot F(n))$ produce the same (co)localization functors, that is, $L=L'$.
\end{proof}

\begin{prop}
   The square (\ref{firstsquare}) commutes, i.e., the natural map \[\alpha\colon L\1_\D \odot L_{n-1}(-) \Rightarrow L(\Lambda\1_\D \whodot -)\]
   of \cref{lax} is an equivalence. In particular, the lax map of recollements is in fact a strict map.
\end{prop}   

\begin{proof}
    By the definition of $\alpha$, one sees that $\alpha$ is an equivalence if and only if $\widetilde{\alpha}$ is an equivalence, as the verticals in (\ref{firstsquare}) are equivalences. So we check that $\widetilde{\alpha}$ is an equivalence. For $A \in \Lambda\D$, the map $\widetilde{\alpha}$ was defined to be the composite
    \[L_{n-1}\Hom_\D^\Sp(\1_\D, A) \xrightarrow{L_{n-1}\Hom_\D^\Sp(\1_\D, \eta_L)} L_{n-1}\Hom_\D^\Sp(\1_\D, LA) \xrightarrow[\simeq]{\varepsilon} \Hom_\D^\Sp(\1_\D, LA)\]
    using \cref{localhoms} for the latter map. In order to check that the first map is an equivalence, it suffices to check that $L_{n-1}\Hom_\D^\Sp(\1_\D, \Gamma A) \simeq 0.$ 

    We have $L_{n-1}\Hom_\D^\Sp(\1_\D,\Gamma A) \simeq \Hom_\D^\Sp(\1_\D,\Gamma A) \otimes L_{n-1}S^0$ as $L_{n-1}$ is smashing. 
    For any $Y \in \Sp$, there is a natural map \[\theta_{ Y}\colon\Hom_\D^\Sp(\1_\D,\Gamma A) \otimes Y \to \Hom_\D^\Sp(\1_\D, \Gamma A \odot Y)\] adjoint to the evaluation map \[\1_\D \odot (\Hom_\D^\Sp(\1_\D,\Gamma A) \otimes Y) \simeq (\1_\D \odot \Hom_\D^\Sp(\1_\D,\Gamma A)) \odot Y \to \Gamma A \odot Y\] where the equivalence comes from (\ref{tensoroftensor}). 
    The set of spectra $Y$ for which $\theta_{Y}$ is an equivalence is a localizing subcategory and clearly contains $S^0$. Therefore $\theta_{Y}$ is an equivalence for all $Y$, in particular, for $Y = L_{n-1}S^0$. By combining this with \cref{lem:iffnthengood}, we have \[L_{n-1}\Hom_\D^\Sp(\1_\D, \Gamma A) \simeq \Hom_\D^\Sp(\1_\D, \Gamma A \odot L_{n-1}S^0) \simeq \Hom_\D^\Sp(\1_\D, L\Gamma A) \simeq 0\] as required. Therefore $\widetilde{\alpha}$ is an equivalence, and hence so is $\alpha$.
\end{proof}

We have now resolved all the steps of \cref{strat}, and as such, applying \cref{thm:equivrec} shows that the induced functor $F \simeq \1_{\D} \odot - \colon L_n\Sp \to \D$ is a symmetric monoidal equivalence. This completes the proof of \cref{main}, which states that $L_n \Sp$ is strongly tt-rigid if $L_{n-1}\Sp$ and $L_{K(n)}\Sp$ are unitally tt-rigid.



\newpage

\begin{landscape}
\pagestyle{empty}
\begin{figure}
    \centering
\[
\resizebox{\displaywidth}{!}{
\xymatrix@C=2em@R=8em{
L\1 \odot L_{n-1}S^0 \ar[dd]|{\rho} \ar[rrrrrr]|{\,L\1 \odot L_{n-1}\eta_{\Lambda} \,} &                                                   &                                            &                    &                                           &                    & L\1 \odot L_{n-1}L_{K(n)}S^0 \ar[d]|{L\1 \odot L_{n-1} \mathrm{coev}_{\Lambda}}            \\
                                  & L(\1 \odot S^0) \ar@{<-}[lu]|{\, \beta \,} \ar[ld]|{\, \simeq \,} \ar[d]|{L\eta_{\Lambda}} \ar[r] \ar@<-1ex>@{}[r]_-{L(\1 \odot \eta_{\Lambda})} & L(\1 \odot L_{K(n)}S^0) \ar[d]|{L \eta_{\Lambda}} \ar@{<-}[rrrru]|{\,\beta\,} \ar[rr]|{\,L(\1 \odot \mathrm{coev}_{\Lambda})\,} &                    & L(\1 \odot \Hom(\1,\Lambda \1 \whodot L_{K(n)}S^0)) \ar[rd]|{\,L(\1 \odot \Hom(\1,\eta_L))  \,} \ar[ld]|{\, L \eta_{\Lambda}\,} \ar[ddd]|{L \mathrm{ev}} &                    & L\1 \odot L_{n-1} \Hom(\1, \Lambda \1 \whodot L_{K(n)}S^0) \ar[d]|{L\1 \odot L_{n-1} \Hom(\1, \eta_L)} \ar[ll]|{\, \beta \,} \\
L\1 \ar[dd]|{L\eta_{\Lambda}}                & L \Lambda(\1 \odot S^0) \ar[ldd]|{\, \simeq \, } \ar[r] \ar@{}@<1ex>[r]^-{L\Lambda(\1 \odot \eta_{\Lambda})}                   & L \Lambda(\1 \odot L_{K(n)}S^0) \ar[d]|{L\gamma} \ar[r] \ar@<-1ex>@{}[r]_-{L\Lambda ( \1 {\odot} \mathrm{coev}_{\Lambda})}                 & L \Lambda (\1 \odot \Hom(\1, \Lambda \1 \whodot L_{K(n)}S^0)) \ar[d]|{L\gamma}  &                                           & L(\1 \odot \Hom(\1,L(\Lambda \1 \whodot L_{K(n)}S^0)))  \ar[d]|{L \mathrm{ev}}  \ar@{<-}[dr]|{\, \vartheta \,} & L\1 \odot L_{n-1} \Hom(\1,L( \Lambda \1 \whodot L_{K(n)}S^0)) \ar[d]|{L\1 \odot \varepsilon_L} \ar[l]|{\,\beta\,}  \\
                                  &                                                   &  L(\Lambda\1 \whodot L_{K(n)}S^0)                      \ar@{=}[rrd]|{\,\mathrm{id}\,} \ar[r] \ar@<1ex>@{}[r]^-{L (\Lambda \1 \widehat{\odot} \mathrm{coev}_{\Lambda})}  \ar[lld]|{L\xi}  & L(\Lambda\1 \whodot\Hom(\1, \Lambda \1 \whodot L_{K(n)}S^0)) \ar[rd]|{\, L \mathrm{ev}_{\Lambda} \, } &                                           & L(L(\Lambda \1 \whodot L_{K(n)} S^0)) \ar[ld]|{\, \varepsilon_{L} \,} & L\1 \odot\Hom(\1,L( \Lambda \1 \whodot L_{K(n)}S^0)) \ar[lld]|{\,\mathrm{ev}_L \,}     \ar[d]|{\, \mathrm{ev}_L \,}     \\
L \Lambda \1                           &                                                   &                                            &                    & L(\Lambda\1 \whodot L_{K(n)}S^0) \ar[llll]|{\, L \xi \,}    \ar@{=}[rr]|{\, \mathrm{id}\,}                  &                    & L(\Lambda\1 \whodot L_{K(n)}S^0)                     
}
}
\]
\captionsetup{singlelinecheck=off}
\caption[optional]{A diagram providing the commutativity of the required square in the proof of \cref{lem:Funit}. We note that the outside square of the above diagram is exactly the square appearing in \cref{lem:Funit} once the definition of $\alpha$ has been spelt out. Here:
\begin{itemize}
    \item $\gamma$ is the natural equivalence $\gamma \colon \Lambda(\1 \odot X) \xrightarrow{\sim} \Lambda \1 \whodot X$ for all $X \in \Sp$,
    \item $\vartheta$ is the natural equivalence $\vartheta \colon LA \odot X \xrightarrow{\sim} L(A \odot X) $ for all $A \in \D$ and $X \in \Sp$ from \cref{tensorandtensor},
    \item $\beta$ is the natural equivalence $\beta \colon LA \odot L_{n-1} Y \xrightarrow{\sim} L(A \odot Y)$ for all $A \in \D$ and $Y \in \Sp$ obtained by composing $\vartheta$ and (\ref{tensoroftensor}),
    \item $\rho$ and $\xi$ are the maps as in  (\ref{units}).
\end{itemize}
The result follows by checking that each inner diagram commutes, which follows by the naturality of the functors involved.
}\label{bigdiagram}
\end{figure}
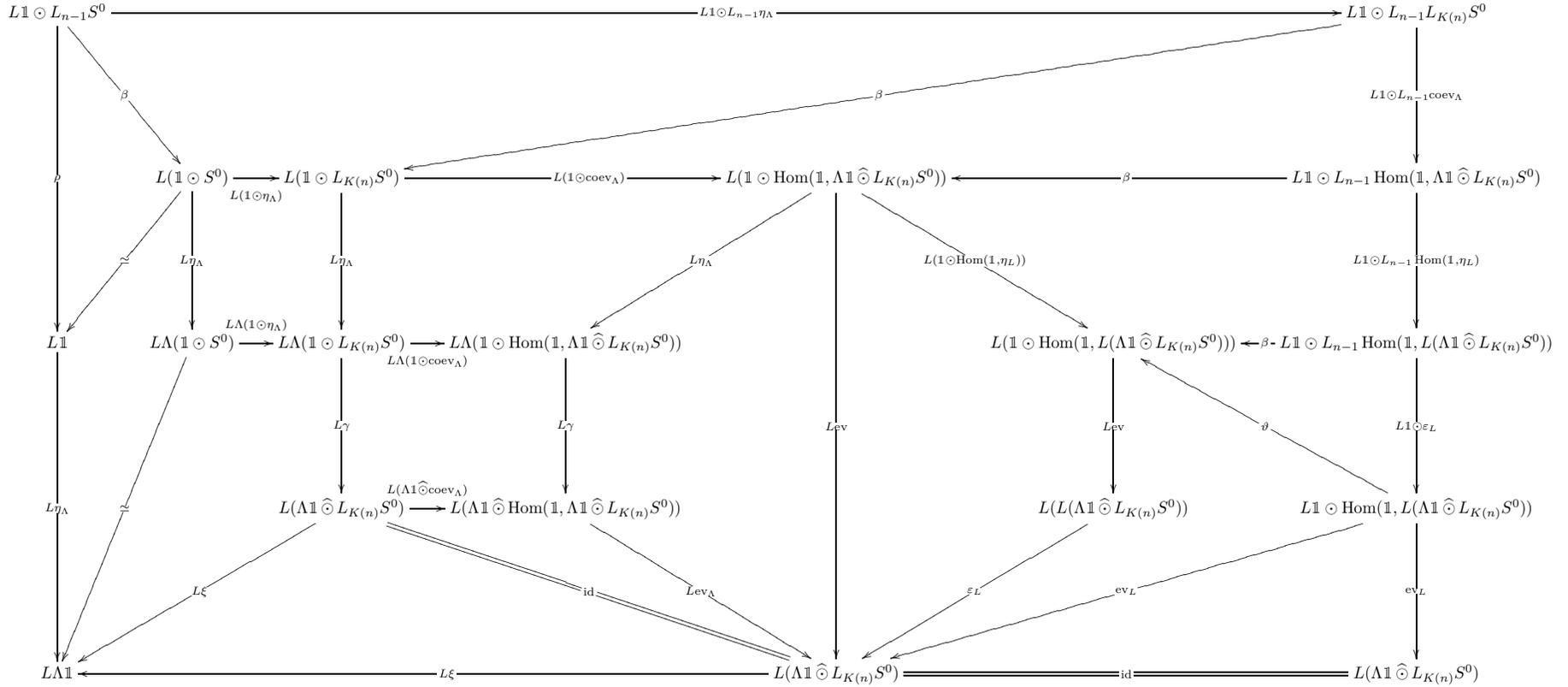
\end{landscape}

\bibliographystyle{abbrv}
\bibliography{brwreferences}
\end{document}